\documentclass[12pt,reqno]{amsart}

\usepackage{latexsym}
\usepackage{amssymb}
\usepackage{mathrsfs}
\usepackage{amsmath}
\usepackage{fancybox,color}
\usepackage{enumerate}
\usepackage[latin1]{inputenc}

\def\1{\raisebox{2pt}{\rm{$\chi$}}}

\newtheorem{theorem}{Theorem}[section]

\newtheorem{lemma}[theorem]{Lemma}
\newtheorem{proposition}[theorem]{Proposition}
\newtheorem{definition}[theorem]{Definition}
\newtheorem{remark}[theorem]{Remark}

\newcommand{\R}{{\mathbb R}}

\def\1{\raisebox{2pt}{\rm{$\chi$}}}

\def\vint_#1{\mathchoice%
         {\mathop{\kern 0.2em\vrule width 0.6em height 0.69678ex depth -0.58065ex
                 \kern -0.8em \intop}\nolimits_{\kern -0.4em#1}}%
         {\mathop{\kern 0.1em\vrule width 0.5em height 0.69678ex depth -0.60387ex
                 \kern -0.6em \intop}\nolimits_{#1}}%
         {\mathop{\kern 0.1em\vrule width 0.5em height 0.69678ex
             depth -0.60387ex
                 \kern -0.6em \intop}\nolimits_{#1}}%
         {\mathop{\kern 0.1em\vrule width 0.5em height 0.69678ex depth -0.60387ex
                 \kern -0.6em \intop}\nolimits_{#1}}}
\def\vintslides_#1{\mathchoice%
         {\mathop{\kern 0.1em\vrule width 0.5em height 0.697ex depth -0.581ex
                 \kern -0.6em \intop}\nolimits_{\kern -0.4em#1}}%
         {\mathop{\kern 0.1em\vrule width 0.3em height 0.697ex depth -0.604ex
                 \kern -0.4em \intop}\nolimits_{#1}}%
         {\mathop{\kern 0.1em\vrule width 0.3em height 0.697ex depth -0.604ex
                 \kern -0.4em \intop}\nolimits_{#1}}%
         {\mathop{\kern 0.1em\vrule width 0.3em height 0.697ex depth -0.604ex
                 \kern -0.4em \intop}\nolimits_{#1}}}

\newcommand{\intav}{\vint}
\newcommand{\aveint}[2]{\mathchoice%
         {\mathop{\kern 0.2em\vrule width 0.6em height 0.69678ex depth -0.58065ex
                 \kern -0.8em \intop}\nolimits_{\kern -0.45em#1}^{#2}}%
         {\mathop{\kern 0.1em\vrule width 0.5em height 0.69678ex depth -0.60387ex
                 \kern -0.6em \intop}\nolimits_{#1}^{#2}}%
         {\mathop{\kern 0.1em\vrule width 0.5em height 0.69678ex depth -0.60387ex
                 \kern -0.6em \intop}\nolimits_{#1}^{#2}}%
         {\mathop{\kern 0.1em\vrule width 0.5em height 0.69678ex depth -0.60387ex
                 \kern -0.6em \intop}\nolimits_{#1}^{#2}}}

\newcommand{\dist}{\operatorname{dist}}

\title[Local maximal operators]{Local maximal operators on fractional Sobolev spaces}

\author[H.\! Luiro]{Hannes Luiro}   
\address[H.L.]{Department of Mathematics and Statistics, P.O. Box 35, FI-40014 University of Jyv\"askyl\"a, Finland}
\email{hannes.s.luiro@jyu.fi}

\author{Antti V. V\"ah\"akangas}
\address[A.V.V]{Department of Mathematics and Statistics, 
Gustaf H\"allstr\"omin katu 2B, FI-00014 University of Helsinki, Finland}
\email{antti.vahakangas@helsinki.fi}

\date{\today}

\begin{document}

\keywords{Local maximal operator, fractional Sobolev space, Hardy inequality}
\subjclass[2010]{42B25, 46E35, 47H99}

\begin{abstract}
In this note we establish the boundedness properties of local maximal operators $M_G$
on the fractional Sobolev spaces $W^{s,p}(G)$  whenever $G$ is an open set in $\R^n$, $0<s<1$ and $1<p<\infty$.
As  an application, we characterize the fractional $(s,p)$-Hardy
inequality on a bounded open set $G$ by a  Maz'ya-type testing condition
localized to Whitney cubes.
\end{abstract}

\maketitle

\markboth{\textsc{H. Luiro  and A. V. V\"ah\"akangas}}
{\textsc{Local maximal operators}}

\section{Introduction}

The
local Hardy--Littlewood maximal operator $M_G=f\mapsto M_G f$ is defined for an open set 
$\emptyset\not=G\subsetneq\R^n$ and a function $f\in L^1_{\textup{loc}}(G)$ by
\[
M_G f(x)=\sup_{r}\intav_{B(x,r)}\lvert f(y)\rvert\,dy\,,\qquad x\in G\,,
\]
where the supremum ranges over $0<r<\dist(x,\partial G)$.
Whereas the (local) Hardy--Littlewood maximal operator is often used to estimate the absolute size, its
Sobolev mapping properties are perhaps less known.
The classical Sobolev regularity of $M_G$ is established by Kinnunen and Lindqvist in \cite{MR1650343};
we also refer to \cite{MR2041705,MR1469106,MR1979008,MR1951818,MR2280193}.
Concerning smoothness of fractional order, the first author established in \cite{MR2579688} the boundedness and continuity properties of $M_G$ 
on the Triebel--Lizorkin
spaces $F^{s}_{pq}(G)$ whenever $G$ is an open set in $\R^n$,  $0<s<1$ and $1<p,q<\infty$. 

Our main focus lies in the mapping properties of $M_G$ on a fractional Sobolev space $W^{s,p}(G)$
with $0<s<1$ and $1<p<\infty$,
cf. Section \ref{s.notation} for the definition or \cite{MR2944369} for a survey of this space.
The intrinsically defined function space $W^{s,p}(G)$ on a given domain $G$  coincides with the trace space $F^{s}_{pp}(G)$ if and only if $G$ is regular, i.e., 
\[
\lvert B(x,r)\cap G\rvert \simeq r^n
\]
whenever $x\in G$ and $0<r<1$,  see \cite[Theorem 1.1]{Z} and \cite[pp. 6--7]{MR1163193}. 
As a consequence, 
if $G$ is a regular domain then $M_G$ is bounded on $W^{s,p}(G)$.
Moreover, the following question arises: is
$M_G$ a bounded operator on $W^{s,p}(G)$ even if $G$ is not regular, e.g., if $G$ has an exterior cusp ?
Our main result provides an affirmative answer to the last question:

\begin{theorem}\label{t.m.bounded}
Let $\emptyset\not=G\subsetneq \R^n$ be an open set, $0<s < 1$ and $1<p < \infty$. Then, there is a constant
$C=C(n,p,s)>0$ such that inequality
\begin{equation}\label{e.max_bdd}
\int_G \int_G \frac{\lvert M_G f(x)-M_G f(y)\rvert^p}{\lvert x-y\rvert^{n+sp}}\,dy\,dx
\le C \int_G \int_G \frac{ \lvert f(x)-f(y)\rvert^p}{\lvert x-y\rvert^{n+sp}}\,dy\,dx
\end{equation}
holds for every $f\in L^p(G)$. In particular, the local Hardy--Littlewood maximal operator $M_G$ is bounded
on the fractional Sobolev space $W^{s,p}(G)$.
\end{theorem}

The relatively simple proof of Theorem \ref{t.m.bounded} is based on a pointwise inequality
in $\R^{2n}$, see  Proposition \ref{p.maximal}. 
That is, for  $f\in L^p(G)$ we  define an auxiliary function $S(f):\R^{2n}\to \R$
\[
S(f)(x,y)=\frac{ \chi_G(x)\chi_G(y) \lvert f(x)-f(y)\rvert}{\lvert x-y\rvert^{\frac{n}{p}+s}}\,,\qquad \text{a.e. }(x,y)\in \R^{2n}\,.
\]
Observe that the $L^p(\R^{2n})$-norm of $S(f)$ coincides with 
$\lvert f\rvert_{W^{s,p}(G)}$,
compare to  definition \eqref{e.semi}.
The key step is to show that $S(M_G f)(x,y)$ is pointwise almost everywhere dominated
by  \[C(n,p,s)\sum_{i,j,k,l\in \{0,1\}}\big(M_{ij}(M_{kl}(Sf))(x,y)+ M_{ij}(M_{kl}(Sf))(y,x)\big)\,,\]
where each  $M_{ij}$ and $M_{kl}$ is either $F\mapsto \lvert F\rvert$ or 
a $V$-directional maximal operator in $\R^{2n}$
that is defined in terms of a fixed $n$-dimensional subspace 
$V\subset \R^{2n}$,  we refer to Definition \eqref{e.lower_dim}. 
The geometry of the open set $G$ does
not have a pivotal role, hence, we are
able to prove the pointwise domination without imposing additional restrictions on $G$.
Theorem \ref{t.m.bounded}
is then a consequence of the fact that the compositions
$M_{ij}M_{kl}$
are bounded on $L^p(\R^{2n})$ if $1<p<\infty$.
The described transference of the problem to the $2n$-dimensional
Euclidean space is a typical step when dealing with norm estimates
for the spaces $W^{s,p}(G)$, we refer
to \cite{E-HSV,ihnatsyeva3,Z} for other  examples.
We plan to adapt the transference method to
norm estimates on
intrinsically defined Triebel--Lizorkin and Besov function spaces on open sets, \cite{MR1163193}. 

As an application of our main result, Theorem \ref{t.m.bounded}, we study
 fractional Hardy inequalities. Let us recall that
an open set $\emptyset\not=G\subsetneq\R^n$ admits an $(s,p)$-Hardy
inequality, for $0<s<1$ and $1<p<\infty$, if there exists a
constant $C>0$ such that inequality 
\begin{equation}\label{e.hardy}
\int_{G} \frac{\lvert f(x)\rvert^p}{\dist(x,\partial G)^{sp}}\,dx
\le   C \int_{G} \int_{G}
\frac{\lvert f(x)-f(y)\rvert ^p}{\lvert x-y\rvert ^{n+sp}}\,dy\,dx
\end{equation}
holds for
all functions $f\in C_c(G)$. 
These inequalities have attracted some interest recently, we
refer to \cite{Dyda3,Dyda2,E-HSV,ihnatsyeva3,ihnatsyeva2,ihnatsyeva1} and the references therein.

In Theorem \ref{t.second} we 
answer a question from \cite{Dyda3}, i.e., we characterize 
those bounded open sets which admit an $(s,p)$-Hardy inequality. The characterization
is given in terms of a localized Maz'ya-type testing condition, where
a lower bound  $\ell(Q)^{n-sp}\lesssim \mathrm{cap}_{s,p}(Q,G)$ for the fractional $(s,p)$-capacities  of all Whitney cubes $Q\in\mathcal{W}(G)$ is
required  and a  quasiadditivity
property of the same capacity is assumed with respect
to all finite families of Whitney cubes.
Aside from  inequality \eqref{e.max_bdd} an important ingredient in the proof of Theorem \ref{t.second} is
the estimate
\begin{equation}\label{e.harnack}
\intav_{2^{-1}Q} f\,dx \le C\, \inf_Q M_Gf\,,
\end{equation}
which holds for a constant  $C>0$ that is
independent of both $Q\in\mathcal{W}(G)$ and $f\in C_c(G)$.
Inequality~\eqref{e.harnack} allows us to circumvent the (apparently unknown) weak Harnack inequalities
for the minimizers that are associated with the $(s,p)$-capacities. The weak Harnack based approach is taken up 
in  \cite{MR3189220}; therein the counterpart of Theorem \ref{t.second} is obtained in case of the classical Hardy inequality, i.e., 
for the gradient instead of the fractional Sobolev seminorm.

The structure of this paper is as follows. In Section \ref{s.notation} we present
the notation and recall various maximal operators. The proof of Theorem \ref{t.m.bounded} is taken up in Section \ref{s.estimates}.
Finally, in Section \ref{s.application}, we give an application of our main result by characterizing
fractional $(s,p)$-Hardy inequalities on bounded open sets.

\section{Notation and preliminaries}\label{s.notation}

\subsection*{Notation}
The open ball centered at $x\in \R^n$ and with radius $r>0$ is  written as $B(x,r)$.
The Euclidean
distance from $x\in\R^n$ to a set $E$ in $\R^n$ is written as $\dist(x,E)$.
The Euclidean diameter of $E$  is $\mathrm{diam}(E)$.
The Lebesgue $n$-measure of a  measurable set $E$ is denoted by $\vert E\vert.$
The characteristic function of a set $E$ is written as $\chi_E$.
We write $f\in C_c(G)$ if $f:G\to \R$ is a continuous function with
 compact support in an open set $G$.
 We let $C(\star,\dotsb,\star)$  denote a positive constant which depends on the quantities appearing
in the parentheses only.

For an open set $\emptyset\not=G\subsetneq \R^n$ in $\R^n$, we let
$\mathcal{W}(G)$ be its Whitney decomposition. 
For the properties of Whitney cubes  we refer to 
\cite[VI.1]{MR0290095}. In particular, we need the inequalities
\begin{equation}\label{dist_est}
\mathrm{diam}(Q)\le \mathrm{dist}(Q,\partial G)\le 4 \mathrm{diam}(Q)\,,\quad Q\in \mathcal{W}(G)\,.
\end{equation}
The center of a cube $Q\in\mathcal{W}(G)$
is written as $x_Q$ and  $\ell(Q)$ is its side length.
By $tQ$, $t>0$, we mean
a cube whose sides are parallel to those of $Q$ and that is centered 
at $x_Q$ and whose
side length is $t\ell(Q)$.

 Let $G$ be an open set in $\R^n$. Let $1< p<\infty$ and $0<s<1$ be given. We write
\begin{equation}\label{e.semi}
\lvert f \rvert_{W^{s,p}(G)} = \bigg( \int_G\int_{G}\frac{\lvert f(x)-f(y)\rvert^p}{\lvert x-y\rvert^{n+s p}}\,
dy\,dx\,\bigg)^{1/p}
\end{equation}
for measurable functions $f$ on $G$ that are finite almost everywhere.
By
$W^{s,p}(G)$ we mean the fractional Sobolev space of functions $f$ in
$L^p({G})$ with
\[\lVert f\rVert_{W^{s,p}({G})}
=\lVert f\rVert_{L^p({G})}+|f|_{W^{s,p}({G})}<\infty\,.\]

\subsection*{Maximal operators}
Let $\emptyset\not=G\subsetneq \R^n$ be an open set.
The  local Hardy--Littlewood maximal function of $f\in L^1_{\textup{loc}}(G)$
is defined as follows. For every $x\in G$,  we write
\begin{equation}\label{d.max}
M_G f(x) = \sup_r \intav_{B(x,r)} \lvert f(y)\rvert\,dy\,,
\end{equation}
where the supremum ranges over $0< r<  \dist(x,\partial G)$.
 For notational convenience, we write
\begin{equation}\label{e.zero}
\intav_{B(x,0)} \lvert f(y)\rvert \,dy = \lvert f(x)\rvert
\end{equation}
whenever $x\in G$ is a Lebesgue point of $\lvert f\rvert$.
It is clear that, at the Lebesgue points of $\lvert f\rvert$, the
supremum in \eqref{d.max} can equivalently be taken over $0\le r\le \dist(x,\partial G)$.

The following lemma is from \cite[Lemma 2.3]{Dyda3}.

\begin{lemma}\label{l.continuity}
Let $\emptyset\not=G\subsetneq \R^n$ be an  open set and $f\in C_c(G)$. Then 
$M_G f$  
is continuous on $G$.
\end{lemma}

Let us fix $i,j\in \{0,1\}$ and $1<p<\infty$. For a function  $F\in L^p(\R^{2n})$ we write
\begin{equation}\label{e.lower_dim}
M_{ij}(F)(x,y) = \sup_{r>0}\intav_{B(0,r)} \lvert F(x+iz,y+jz)\rvert \,dz
\end{equation}
for almost every $(x,y)\in \R^{2n}$.
Observe that $M_{00}(F)=\lvert F\rvert$. 
By applying Fubini's theorem in suitable coordinates
and boundedness of the centred Hardy--Littlewood maximal operator in $L^p(\R^n)$ we find that $M_{ij}=F\mapsto M_{ij}(F)$ is 
a bounded operator on $L^p(\R^{2n})$; let us remark that the measurability
of $M_{ij}(F)$ for a given $F\in L^p(\R^{2n})$ can be checked by first noting that the supremum in \eqref{e.lower_dim} can
be restricted to the rational numbers $r>0$
and then adapting
the proof of \cite[Theorem 8.14]{MR924157} with each $r$ separately.

\section{The proof of Theorem \ref{t.m.bounded}}\label{s.estimates}

Within this section we prove our main result, namely Theorem \ref{t.m.bounded} that is
stated in the Introduction.
Let us first  recall a convenient notation.
Namely, for $f\in L^p(G)$ we write
\[
S(f)(x,y) = S_{G,n,s,p}(f)(x,y)=\frac{ \chi_G(x)\chi_G(y) \lvert f(x)-f(y)\rvert}{\lvert x-y\rvert^{\frac{n}{p}+s}}
\]
for almost every $(x,y)\in \R^{2n}$.
The main tool for proving Theorem \ref{t.m.bounded} is a  pointwise inequality, stated in Proposition \ref{p.maximal}, which might be
of independent interest.

\begin{proposition}\label{p.maximal}
Let $\emptyset\not=G\subsetneq\R^n$ be an open set, $0<s<1$ and $1<p<\infty$. Then there exists a constant $C=C(n,p,s)>0$ such that,
for almost every $(x,y)\in\R^{2n}$, inequality
\begin{equation}\label{e.dom}
S(M_G  f)(x,y)
\le C\sum_{i,j,k,l\in \{0,1\}} \big( M_{ij}(M_{kl}(S f))(x,y) +M_{ij}(M_{kl}(S f))(y,x)\big)
\end{equation}
holds whenever  $f\in L^p(G)$ and $S f\in L^p(\R^{2n})$.
\end{proposition}

By postponing the proof of Proposition \ref{p.maximal} for a while, we can prove  Theorem \ref{t.m.bounded}.

\begin{proof}[Proof of Theorem \ref{t.m.bounded}]
Fix $f\in L^p(G)$.
Without loss of generality, we may assume that the right hand side of inequality \eqref{e.max_bdd} is finite.
Hence $S f \in L^p(\R^{2n})$ and
inequality \eqref{e.max_bdd} is a consequence of Proposition \ref{p.maximal} and the boundedness
of maximal operators $M_{ij}$ on $L^p(\R^{2n})$.
\end{proof}

We proceed to the postponed proof that is motivated by  that of \cite[Theorem 3.2]{MR2579688}.

\begin{proof}[Proof of Proposition \ref{p.maximal}]
By replacing the function $f$ with $\lvert f\rvert$ we may assume that
 $f\ge 0$. Since $f\in L^p(G)$ and, hence, $M_G f\in L^p(G)$ we may restrict
 ourselves to points $(x,y)\in G\times G$ for which both 
 $x$ and $y$ are Lebesgue points of $f$ and both $M_Gf(x)$ and $M_Gf(y)$ are finite. Moreover,
by symmetry,  we may further assume that  $M_G f(x)>M_G f(y)$. 
These reductions allow us to
find $0\le r(x)\le \dist(x,\partial G)$ and $0\le r(y)\le \dist(y,\partial G)$ such that the estimate
\begin{align*}
S(M_G f)(x,y)&=\frac{\lvert M_G f(x)-M_G f(y)\rvert}{\lvert x-y\rvert^{\frac{n}{p}+s}}\\
&=\frac{\lvert \intav_{B(x,r(x))}f\,-\intav_{B(y,r(y))}f\,\rvert}{\lvert x-y\rvert^{\frac{n}{p}+s}}
\leq\frac{\lvert \intav_{B(x,r(x))}f\,-\intav_{B(y,r_2)}f\,\rvert}{\lvert x-y\rvert^{\frac{n}{p}+s}}
\end{align*}
is valid for any given number \[0\le r_2\le \dist(y,\partial G)\,;\] 
this number will
be chosen in a convenient manner in the two case studies below.

\smallskip

\noindent
{\bf{Case $r(x)\le \lvert x-y\rvert$}.}
Let us denote $r_1=r(x)$ and choose
\begin{equation}\label{e.case}
r_2=0\,.
\end{equation}
If $r_1=0$, then we get from \eqref{e.case} and our notational convention \eqref{e.zero} that \[S(M_G f)(x,y)\leq S(f)(x,y)\,.\]
Suppose then that $r_1>0$. Now
\begin{align*}
S(M_G f)(x,y)&\le \frac{1}{\lvert x-y\rvert^{\frac{n}{p}+s}}\bigg{|}\intav_{B(x,r_1)}f(z)\,dz\,-\intav_{B(y,r_2)}f(z)\,dz\,\bigg{|}  \\&=\frac{1}{\lvert x-y\rvert^{\frac{n}{p}+s}} \bigg{|}\intav_{B(x,r_1)} f(z)-f(y)\,dz\,\,\bigg{|}\\
& \lesssim \intav_{B(0,r_1)} \frac{\chi_G(x+z)\chi_G(y)\lvert f(x+z)-f(y)\rvert}{\lvert x+z-y\rvert^{\frac{n}{p}+s}}\,dz\le M_{10}(S f)(x,y)\,.
\end{align*}
We have shown that
\begin{align*}
S(M_G f)(x,y)
 \lesssim S (f)(x,y) + M_{10}(S f)(x,y)
\end{align*}
and it is clear that inequality \eqref{e.dom} follows (recall that $M_{00}$ is the identity
operator when restricted to non-negative functions).

\smallskip
\noindent
{\bf{Case $r(x)>\lvert x-y\rvert$}.}
Let us denote $r_1=r(x)>0$ and choose
\[
r_2 = r(x) - \lvert x-y\rvert>0\,.
\]
We then have
\begin{align*}
&\bigg{|}\intav_{B(x,r_1)}f(z)\,dz\,-\intav_{B(y,r_2)}f(z)\,dz\,\bigg{|}
=\bigg{|}\intav_{B(0,r_1)}f(x+z)-f(y+\frac{r_2}{r_1}z)\,dz\,\bigg{|}\\
&=\bigg{|}\intav_{B(0,r_1)}\bigg(f(x+z)-\intav_{B(y+\frac{r_2}{r_1}z,2\lvert x-y\rvert)\cap G}f(a)\,da\bigg)\\ &\qquad \qquad+\bigg(\intav_{B(y+\frac{r_2}{r_1}z,2\lvert x-y\rvert)\cap G}f(a)\,da
-f(y+\frac{r_2}{r_1}z)\bigg)\,dz\,\bigg{|}\\
&\leq A_1+A_2\,,
\end{align*}
where we have written
\begin{align*}
A_1 &= \intav_{B(0,r_1)}\bigg(\intav_{B(y+\frac{r_2}{r_1}z,2\lvert x-y\rvert)\cap G}|f(x+z)-f(a)|\,da\,\bigg)\,dz\,,\\
A_2&=\intav_{B(0,r_1)}\bigg(\intav_{B(y+\frac{r_2}{r_1}z,2\lvert x-y\rvert)\cap G}|f(y+\frac{r_2}{r_1}z)-f(a)|\,da\,\bigg)\,dz\,.
\end{align*}
We  estimate both of these terms separately,
but first we need certain auxiliary estimates.

\smallskip
 Recall that $r_2=r_1-\lvert x-y\rvert$.
Hence, for every $z\in B(0,r_1)$,
\begin{align*}
\lvert y+\frac{r_2}{r_1}z-(x+z)\rvert&=\lvert y-x+\frac{(r_2-r_1)}{r_1} z\rvert\\
&\leq \lvert y-x\rvert+\frac{\lvert x-y\rvert}{r_1}\lvert z\rvert\leq 2\lvert y-x\rvert\,.
\end{align*}
This, in turn, implies that 
\begin{equation}\label{e.inclusion}
B(y+\frac{r_2}{r_1}z,2\lvert x-y\rvert)\subset B(x+z,4\lvert x-y\rvert)
\end{equation}
whenever $z\in B(0,r_1)$.
Moreover, since $r_1> \lvert x-y\rvert$ and $\{y+\frac{r_2}{r_1}z, x+z\}\subset B(x,r_1)\subset G$ 
if $\lvert z\rvert< r_1$,
we obtain the two equivalences
\begin{equation}\label{e.equivalence}
\lvert B(y+\frac{r_2}{r_1}z,2\lvert x-y\rvert)\cap G\rvert \simeq \lvert x-y\rvert^n\simeq \lvert B(x+z,4\lvert x-y\rvert)\cap G\rvert
\end{equation}
for every $z\in B(0,r_1)$. Here the implied constants depend only on $n$.

\smallskip
\noindent
{\bf An estimate for $A_1$}.
The inclusion \eqref{e.inclusion} and inequalities \eqref{e.equivalence} show that, in the definition of $A_1$, we can replace the domain of integration in the  inner integral by $B(x+z,4\lvert x-y\rvert)\cap G$ and, at the same time,  control the error term while integrating on average. That is to say,     
\begin{align*}
A_1&\lesssim \intav_{B(0,r_1)}\bigg(\intav_{B(x+z,4\lvert x-y\rvert)\cap G}|f(x+z)-f(a)|\,da\,\bigg)\,dz\,.
\end{align*}
By observing that both $x+z$ and $a$ in the last double integral belong to  $G$ and using \eqref{e.equivalence} again, we can continue as follows:
\begin{align*}
\frac{A_1}{\lvert x-y\rvert^{\frac{n}{p}+s}}
&\lesssim
\intav_{B(0,r_1)}\bigg(\intav_{B(x+z,4\lvert x-y\rvert)}
\frac{ \chi_G(x+z)\chi_G(a)\lvert f(x+z)-f(a)\rvert}{\lvert x+z-a\rvert^{\frac{n}{p}+s}}\,da\bigg)dz\\
&\lesssim\intav_{B(0,r_1)}\bigg(\intav_{B(y+z,5\lvert x-y\rvert)}
S(f)(x+z,a)\,da\,\bigg)\,dz\,.
\end{align*}
Applying the maximal operators defined in Section \ref{s.notation} we find that
\begin{align*}
\frac{A_1}{\lvert x-y\rvert^{\frac{n}{p}+s}}&\lesssim \intav_{B(0,r_1)} M_{01}(Sf)(x+z,y+z)\,dz\leq M_{11}(M_{01}(Sf))(x,y)\,.
\end{align*}

\smallskip
\noindent
{\bf An estimate for $A_2$.} 
We use the inclusion $y+\frac{r_2}{r_1}z\in G$ for all  $z\in B(0,r_1)$ and then apply 
the first equivalence in \eqref{e.equivalence} to obtain
\begin{align*}
A_2 &= \intav_{B(0,r_1)}\bigg(\intav_{B(y+\frac{r_2}{r_1}z,2\lvert x-y\rvert)\cap G}
\chi_G(y+\frac{r_2}{r_1}z)\chi_G(a)\lvert f(y+\frac{r_2}{r_1}z)-f(a)\rvert \,da\,\bigg)\,dz\\
&\lesssim \intav_{B(0,r_1)}\bigg(\intav_{B(y+\frac{r_2}{r_1}z,2\lvert x-y\rvert)}
\chi_G(y+\frac{r_2}{r_1}z)\chi_G(a)\lvert f(y+\frac{r_2}{r_1}z)-f(a)\rvert \,da\,\bigg)\,dz\,.
\end{align*}
Hence, a change of variables yields
\begin{align*}
\frac{A_2}{\lvert x-y\rvert^{\frac{n}{p}+s}}
&\lesssim \intav_{B(0,r_2)}\bigg(\intav_{B(y+z,2\lvert x-y\rvert)}\frac{\chi_G(y+z)\chi_G(a)\lvert f(y+z)-f(a)\rvert}{\lvert y+z-a\rvert^{\frac{n}{p}+s}}\,d
a\,\bigg)\,dz\\
&\lesssim \intav_{B(0,r_2)}\bigg(\intav_{B(x+z,3\lvert x-y\rvert)}S(f)(y+z,a)\,da\,\bigg)\,dz\,.
\end{align*}
Applying operators $M_{01}$ and $M_{11}$ from Section \ref{s.notation}, we can proceed as follows
\begin{align*}
\frac{A_2}{\lvert x-y\rvert^{\frac{n}{p}+s}}&\lesssim\intav_{B(0,r_2)} M_{01}(Sf)(y+z,x+z)\,dz\leq M_{11}(M_{01}(Sf))(y,x)\,.
\end{align*}

\smallskip
Combining the above estimates for $A_1$ and $A_2$ we end up with  
\begin{align*}
S(M_Gf)(x,y)\le\frac{A_1+A_2}{\lvert x-y\rvert^{\frac{n}{p}+s}}  \lesssim M_{11}(M_{01}(S f))(x,y)+M_{11}(M_{01}(S f))(y,x)
\end{align*}
and inequality \eqref{e.dom} follows.
\end{proof}

\section{Application to fractional Hardy inequalities}\label{s.application}

We apply Theorem \ref{t.m.bounded} by
solving a certain localisation problem for $(s,p)$-Hardy inequalities and our result is formulated in Theorem \ref{t.second} below.
Recall that an open set $\emptyset\not=G\subsetneq\R^n$ admits an $(s,p)$-Hardy
inequality, for $0<s<1$ and $1<p<\infty$, if there is a
constant $C>0$ such that inequality 
\begin{equation}
\int_{G} \frac{\lvert f(x)\rvert^p}{\dist(x,\partial G)^{sp}}\,dx
\le   C \int_{G} \int_{G}
\frac{\lvert f(x)-f(y)\rvert ^p}{\lvert x-y\rvert ^{n+sp}}\,dy\,dx
\end{equation}
holds for
all functions $f\in C_c(G)$. 
We need a characterization of $(s,p)$-Hardy inequality in 
terms of
the following  $(s,p)$-capacities
of compact sets $K\subset G$\,; we write
\[
\mathrm{cap}_{s,p}(K,G) = \inf_u \lvert u\rvert_{W^{s,p}(G)}^p\,,
\]
where the infimum is taken over all real-valued functions $u\in C_c(G)$ such that $u(x)\ge 1$ for  every  $x\in K$.
The `Maz'ya-type characterization' stated in Theorem \ref{t.maz'ya} is \cite[Theorem 1.1]{Dyda3}
and it extends to the case $0<p<\infty$.
For information on characterizations of this type, we refer to \cite[Section 2]{MR817985} and \cite{MR2723821}.

\begin{theorem}\label{t.maz'ya}
Let $0<s<1$ and $1<p<\infty$. Then an  open set $\emptyset\not=G\subsetneq \R^n$ admits an $(s,p)$-Hardy inequality if and only if
there is a constant $C>0$ such that
\begin{equation}\label{e.Mazya}
\int_K \mathrm{dist}(x,\partial G)^{-sp}\,dx \le C\,\mathrm{cap}_{s,p}(K,G)
\end{equation}
for every compact set $K\subset G$.
\end{theorem}

We solve a `localisation problem of the testing condition \eqref{e.Mazya}', which is stated as a question in \cite[p. 2]{Dyda3}. Roughly speaking, we prove that
if $\mathrm{cap}_{s,p}(\cdot,G)$ satisfies a quasiadditivity property, see Definition \ref{d.quasi},
then $G$ admits an $(s,p)$-Hardy inequality if and only if
inequality \eqref{e.Mazya} holds for all Whitney cubes $K=Q\in\mathcal{W}(G)$.

\begin{definition}\label{d.quasi}
The $(s,p)$-capacity $\mathrm{cap}_{s,p}(\cdot,G)$ is weakly $\mathcal{W}(G)$-quasiadditive, if
there exists a constant $N>0$ such that 
\begin{equation}\label{e.quasi}
\sum_{Q\in\mathcal{W}(G)} \mathrm{cap}_{s,p}(K\cap Q,G) \le N\,\mathrm{cap}_{s,p}(K, G)
\end{equation}
whenever $K=\bigcup_{Q\in\mathcal{E}}Q$ and $\mathcal{E}\subset\mathcal{W}(G)$ is a finite family
of Whitney cubes.
\end{definition}

More precisely, we prove the following characterization.

\begin{theorem}\label{t.second}
Let $0<s<1$ and $1<p<\infty$ be such that $sp<n$.
Suppose that $G\not=\emptyset$ is a bounded open set in $\R^n$. Then  the following conditions (A) and (B) are equivalent.
\begin{itemize} 
\item[(A)] $G$ admits an $(s,p)$-Hardy inequality;
\item[(B)]  $\mathrm{cap}_{s,p}(\cdot,G)$ is weakly $\mathcal{W}(G)$-quasiadditive  and there exists a constant $c>0$ such that
\begin{equation}\label{e.testing}
\ell(Q)^{n-sp}\le c\,\mathrm{cap}_{s,p}(Q,G)
\end{equation} for every  $Q\in\mathcal{W}(G)$.
\end{itemize}
\end{theorem}

Before the proof of Theorem \ref{t.second}, let us make a remark concerning condition (B).

\begin{remark}
The counterexamples  in \cite[Section 6]{Dyda3} show that neither one of the two conditions
(i.e., weak $\mathcal{W}(G)$-quasiadditivity of the capacity and the lower bound \eqref{e.testing} 
for the capacities of Whitney cubes) appearing in Theorem \ref{t.second}(B) 
is implied by the other one. Accordingly, both of these conditions are needed
for the characterization. 
\end{remark}

\begin{proof}[Proof of Theorem \ref{t.second}]
The implication from (A) to (B) follows from \cite[Proposition 4.1]{Dyda3} in combination with \cite[Lemma 2.1]{Dyda3}.
In the following proof of the implication from (B) to (A) we adapt the argument given in \cite[Proposition 5.1]{Dyda3}.

By Theorem \ref{t.maz'ya}, it suffices to show that
\begin{equation}\label{e.wish}
\int_K \mathrm{dist}(x,\partial G)^{-sp}\, dx \lesssim \mathrm{cap}_{s,p}(K,G)\,,
\end{equation}
whenever $K\subset G$ is compact.
Let us fix a compact set $K\subset G$ and an admissible test function $u$ for $\mathrm{cap}_{s,p}(K,G)$. We partition $\mathcal{W}(G)$ as $\mathcal{W}_1\cup \mathcal{W}_2$, where 
\begin{align*}
\mathcal{W}_1 = \{Q\in\mathcal{W}(G)\,:\, \langle u\rangle_{2^{-1}Q} :=\intav_{2^{-1}Q} u < 1/2\}\,,\qquad
\mathcal{W}_2 = 
\mathcal{W}(G)\setminus \mathcal{W}_1\,.
\end{align*}
Write the left-hand side of \eqref{e.wish} as
\begin{equation}\label{e.w_split}
 \bigg\{\sum_{Q\in\mathcal{W}_1} + \sum_{Q\in\mathcal{W}_2}\bigg\}
\int_{K\cap Q} \mathrm{dist}(x,\partial G)^{-sp}\,dx\,.
\end{equation}
To estimate the first series we observe that, for every $Q\in\mathcal{W}_1$ and every $x\in K\cap Q$,
\[
\tfrac 12=1-\tfrac12 <  u(x) -  \langle u\rangle_{2^{-1}Q}  = \lvert u(x)-\langle u\rangle_{2^{-1}Q}\rvert\,.
\]
Thus, by Jensen's  inequality and \eqref{dist_est},
\begin{align*}
\sum_{Q\in\mathcal{W}_1} 
\int_{K\cap Q} \mathrm{dist}(x,\partial G)^{-sp}\,dx
&\lesssim \sum_{Q\in\mathcal{W}_1} 
\ell(Q)^{-sp} \int_{Q} \lvert u(x)-\langle u\rangle_{2^{-1}Q}\rvert^p\,dx\\
&\lesssim \sum_{Q\in\mathcal{W}_1} 
\ell(Q)^{-n-sp} \int_{Q}\int_{Q} \lvert u(x)-u(y)\rvert^p\,dy\,dx\\
&\lesssim \sum_{Q\in\mathcal{W}_1} 
\int_{Q}\int_{Q} \frac{\lvert u(x)-u(y)\rvert^p}{\lvert x-y\rvert^{n+sp}}\,dy\,dx
\\&\lesssim \lvert u\rvert_{W^{s,p}(G)}^p\,.
\end{align*}

Let us then focus on the remaining series in \eqref{e.w_split}.
Let us consider $Q\in\mathcal{W}_2$ and $x\in Q$.
Observe that $2^{-1}Q\subset B(x,\tfrac 45\mathrm{diam}(Q))$. Hence, by inequalities \eqref{dist_est}, 
\begin{equation}\label{e.m_est}
\begin{split}
M_G u(x) 
\gtrsim \intav_{2^{-1}Q} u(y)\,dy \geq \tfrac 12\,.
\end{split}
\end{equation}
The support of $M_G u$ is a compact set in $G$ by the boundedness of $G$
and the fact that $u\in C_c(G)$. 
By Lemma \ref{l.continuity}, we find that $M_G u$ is continuous. Concluding from these remarks we find that 
there is  $\rho>0$, depending only on $n$, such that $\rho M_G u$ is an admissible test function for 
 $\mathrm{cap}_{s,p}(\cup_{Q\in\mathcal{W}_2} Q,G)$.
The family $\mathcal{W}_2$ is finite, as $u\in C_c(G)$.
Hence, by condition (B) and the inequality \eqref{e.m_est},
\begin{align*}
\sum_{Q\in\mathcal{W}_2} 
\int_{K\cap Q} \mathrm{dist}(x,\partial G)^{-sp}\,dx 
&\lesssim \sum_{Q\in\mathcal{W}_2} 
\ell(Q)^{n-sp}\\
&\le c \sum_{Q\in\mathcal{W}_2} \mathrm{cap}_{s,p}(Q,G) \\
&\le cN \mathrm{cap}_{s,p}\Big(\bigcup_{Q\in\mathcal{W}_2} Q,G\Big)\\
&\le cN\rho^p \int_G \int_G \frac{\lvert M_G u(x)-M_G u(y)\rvert^p}{\lvert x-y\rvert^{n+sp}}\,dy\,dx\,.
\end{align*}
By Theorem \ref{t.m.bounded},  the last term is dominated by 
\[
C(n,s,p,N,c,\rho)\lvert u\rvert_{W^{s,p}(G)}^p\,.
\] 
The desired inequality \eqref{e.wish} follows from the considerations above.
\end{proof}

\def\cprime{$'$} \def\cprime{$'$}


\begin{thebibliography}{10}

\bibitem{MR2944369}
E.~Di~Nezza, G.~Palatucci, and E.~Valdinoci.
\newblock Hitchhiker's guide to the fractional {S}obolev spaces.
\newblock {\em Bull. Sci. Math.}, 136(5):521--573, 2012.

\bibitem{Dyda3}
B.~Dyda and A.~V. V{\"a}h{\"a}kangas.
\newblock Characterizations for fractional {H}ardy inequality.
\newblock {\em Adv. Calc. Var.}, to appear.

\bibitem{Dyda2}
B.~Dyda and A.~V. V{\"a}h{\"a}kangas.
\newblock A framework for fractional {H}ardy inequalities.
\newblock {\em Ann. Acad. Sci. Fenn. Math.}, to appear.

\bibitem{E-HSV}
D.~Edmunds, R.~Hurri-Syrj\"anen, and A.~V. V{\"a}h{\"a}kangas.
\newblock Fractional {H}ardy-type inequalities in domains with uniformly fat
  complement.
\newblock {\em Proc. Amer. Math. Soc.}, 142(3):897--907, 2014.

\bibitem{MR2041705}
P.~Haj{\l}asz and J.~Onninen.
\newblock On boundedness of maximal functions in {S}obolev spaces.
\newblock {\em Ann. Acad. Sci. Fenn. Math.}, 29(1):167--176, 2004.

\bibitem{ihnatsyeva3}
L.~Ihnatsyeva, J.~Lehrb\"ack, H.~Tuominen, and A.~V. V{\"a}h{\"a}kangas.
\newblock {F}ractional {H}ardy inequalities and visibility of the boundary.
\newblock arXiv: 1305.4616.

\bibitem{ihnatsyeva2}
L.~Ihnatsyeva and A.~V. V{\"a}h{\"a}kangas.
\newblock Hardy inequalities in {T}riebel--{L}izorkin spaces {II}. {A}ikawa
  dimension.
\newblock {\em Ann. Mat. Pura Appl.} (4), 2013, (DOI)
  10.1007/s10231-013-0385-z.

\bibitem{ihnatsyeva1}
L.~Ihnatsyeva and A.~V. V{\"a}h{\"a}kangas.
\newblock Hardy inequalities in {T}riebel-{L}izorkin spaces.
\newblock {\em Indiana Univ. Math. J.}, to appear.

\bibitem{MR1469106}
J.~Kinnunen.
\newblock The {H}ardy-{L}ittlewood maximal function of a {S}obolev function.
\newblock {\em Israel J. Math.}, 100:117--124, 1997.

\bibitem{MR2723821}
J.~Kinnunen and R.~Korte.
\newblock Characterizations for the {H}ardy inequality.
\newblock In {\em Around the research of {V}ladimir {M}az'ya. {I}}, volume~11
  of {\em Int. Math. Ser. (N. Y.)}, pages 239--254. Springer, New York, 2010.

\bibitem{MR1650343}
J.~Kinnunen and P.~Lindqvist.
\newblock The derivative of the maximal function.
\newblock {\em J. Reine Angew. Math.}, 503:161--167, 1998.

\bibitem{MR1979008}
J.~Kinnunen and E.~Saksman.
\newblock Regularity of the fractional maximal function.
\newblock {\em Bull. London Math. Soc.}, 35(4):529--535, 2003.

\bibitem{MR1951818}
S.~Korry.
\newblock Boundedness of {H}ardy-{L}ittlewood maximal operator in the framework
  of {L}izorkin-{T}riebel spaces.
\newblock {\em Rev. Mat. Complut.}, 15(2):401--416, 2002.

\bibitem{MR3189220}
J.~Lehrb{\"a}ck and N.~Shanmugalingam.
\newblock Quasiadditivity of {V}ariational {C}apacity.
\newblock {\em Potential Anal.}, 40(3):247--265, 2014.

\bibitem{MR2280193}
H.~Luiro.
\newblock Continuity of the maximal operator in {S}obolev spaces.
\newblock {\em Proc. Amer. Math. Soc.}, 135(1):243--251 (electronic), 2007.

\bibitem{MR2579688}
H.~Luiro.
\newblock On the regularity of the {H}ardy-{L}ittlewood maximal operator on
  subdomains of {$\Bbb R\sp n$}.
\newblock {\em Proc. Edinb. Math. Soc. (2)}, 53(1):211--237, 2010.

\bibitem{MR817985}
V.~G. Maz'ya.
\newblock {\em Sobolev spaces}.
\newblock Springer Series in Soviet Mathematics. Springer-Verlag, Berlin, 1985.
\newblock Translated from the Russian by T. O. Shaposhnikova.

\bibitem{MR924157}
W.~Rudin.
\newblock {\em Real and complex analysis}.
\newblock McGraw-Hill Book Co., New York, third edition, 1987.

\bibitem{MR0290095}
E.~M. Stein.
\newblock {\em Singular integrals and differentiability properties of
  functions}.
\newblock Princeton Mathematical Series, No. 30. Princeton University Press,
  Princeton, N.J., 1970.

\bibitem{MR1163193}
H.~Triebel.
\newblock {\em Theory of function spaces. {II}}, volume~84 of {\em Monographs
  in Mathematics}.
\newblock Birkh\"auser Verlag, Basel, 1992.

\bibitem{Z}
Y.~Zhou.
\newblock Fractional {S}obolev extension and imbedding.
\newblock {\em Trans. Amer. Math. Soc.}, to appear.

\end{thebibliography}
\end{document}